\documentclass[12pt,reqno]{amsart}
\usepackage{%gmj-de,
amssymb}

\def\currentvolume{16}
\def\curryr{09}
\def\currentnumber{2}
\def\currentstartpage{1}
\def\currentendpage{}

\def\abstractname{\bf Abstract}
\usepackage{amsmath,epsfig}
\usepackage{amssymb,amscd,amsfonts}

\usepackage{epic,longtable,enumerate}

\newtheorem{lemma}{\indent Lemma}[section]
\newtheorem{proposition}[lemma]{\indent Proposition}
\newtheorem{theorem}[lemma]{\indent Theorem}

\theoremstyle{definition}

\theoremstyle{remark}
\newtheorem{remark}[lemma]{\indent Remark}

\newcommand{\Q}{\mathbb{Q}}
\newcommand{\C}{\mathbb{C}}
\newcommand{\Z}{\mathbb{Z}}
\newcommand{\GL}{\mathop{\mathrm{GL}}}
\newcommand{\Ind}{\mathop{\mathrm{Ind}}}
\newcommand{\rank}{\mathop{\mathrm{rank}}}
\newcommand{\ad}{\mathop{\mathrm{ad}}}
\newcommand{\g}{\mathop{\mathfrak{g}}}
\newcommand{\p}{\mathop{\mathfrak{p}}}
\newcommand{\ml}{\mathop{\mathfrak{l}}}
\newcommand{\n}{\mathop{\mathfrak{n}}}
\newcommand{\h}{\mathop{\mathfrak{h}}}
\newcommand{\uu}{\mathop{\mathfrak{u}}}
\newcommand{\gl}{\mathop{\mathfrak{gl}}}
\newcommand{\dd}{{\mathop{\mathrm{d}}}}
\newcommand{\D}{\widetilde{D}}
\newcommand{\ssl}{\mathop{\mathfrak{sl}}}
\newcommand{\s}{\mathop{\mathfrak{s}}}
\newcommand{\lab}[1]{\footnotesize{#1}}
\newcommand{\nm}[1]{\footnotesize{#1}}

\begin{document}

\title[Induced nilpotent orbits]{Induced Nilpotent Orbits of the Simple Lie Algebras of Exceptional Type}

\author[W. A. de Graaf]{Willem A. de Graaf}

\author[A. Elashvili]{Alexander Elashvili}

\maketitle

\hfill {\it Dedicated to Nodar Berikashvili}

\hfill {\it on the occasion of his 80th birthday}

\begin{abstract}
We describe algorithms for computing the induced nilpotent orbits in
semisimple Lie algebras. We use them to obtain the induction tables for
the Lie algebras of exceptional type. This also yields the classification of
the rigid nilpotent orbits in those Lie algebras.
\bigskip


\end{abstract}

\section{Introduction}

Let $\g$ be a simple complex Lie algebra, and let $G$ be a connected algebraic
group with Lie algebra $\g$. Then $G$ acts on $\g$, and a natural question is
what the $G$-orbits in $\g$ are.
The nilpotent $G$-orbits in $\g$ have been studied in great detail
(see for example \cite{colmcgov}). They have been
classified in terms of so-called weighted Dynkin diagrams. In \cite{lusp}
the notion of {\em induced nilpotent orbit} was introduced. Let $\p\subset \g$
be a parabolic subalgebra with Levi decomposition $\p = \ml\oplus \n$, where
$\n$ is the nilradical. Let $\mathcal{O}_{\ml}$ be a nilpotent orbit in $\ml$.
Then in \cite{lusp} it is shown that there is a unique nilpotent orbit
$Ge\subset \g$ such that $Ge \cap (\mathcal{O}_{\ml} \oplus \n)$ is dense
in $\mathcal{O}_{\ml} \oplus \n$. The orbit $Ge$ is said to be induced from
$\mathcal{O}_{\ml}$.

Naturally
this led to the question which nilpotent orbits are induced, and which are not.
For the simple Lie algebras of classical type this question was treated
by Spaltenstein (\cite{spaltenstein}) and later
by Kempken (\cite{kempken}).
The same problem for the exceptional types was first solved
by Elashvili. Elashvili (exceptional case) and Spaltenstein (classical case)
announced these results in a joint talk at the 1979 Oberwolfach conference on
Transformation Groups and Invariant Theory. Later (\cite{elashvili}, see also
\cite{spaltenstein}) Elashvili has published tables which, for the Lie algebras
of exceptional type, list for each induced nilpotent orbit
exactly from which data it is induced (a Levi subalgebra, and a nilpotent
orbit in it). We call these lists {\em induction tables}.

It is the objective of this paper to give algorithms that compute the induction
table for a given semisimple Lie algebra (Sections \ref{sect:prelim},
\ref{sect:induce}). We have implemented these algorithms
in the computer algebra system {\sf GAP}4 (\cite{gap4}).
Using this we recomputed Elashvili's tables (and fortunately this confirmed
their correctness). They are given in
Section \ref{sect:tables}. This serves two purposes. Firstly, these
computations constitute an independent check of the correctness of the
tables. Secondly, it is our objective to make the tables more easily available.
A new feature of our tables is that they contain a representative for each
induced orbit. That is a nilpotent element with two properties: it lies in a
particular subalgebra (denoted $\uu(\widetilde{D})$, see Section
\ref{sect:induce}) of the parabolic subalgebra associated with the induction,
and it is a representative of the induced nilpotent orbit.
Also the validity of these representatives has been checked by computer.

\section{Preliminaries}\label{sect:prelim}

In this section we recall some notions from the literature. Secondly, we
describe some basic algorithms that we use, and that we believe to be of
independent interest. Our computational set up is as in \cite{gra6}. In
particular, $\g$ will be a simple complex Lie algebra given by a multiplication
table relative to a Chevalley basis. This means that all structure constants
are integers. Therefore, all computations will take place over the base field
$\Q$.

\subsection{Finding $\ssl_2$-triples}

Let $e\in \g$ be a nilpotent element. Then by the Jacobson-Morozov lemma
(cf. \cite{jac}) there are $f,h\in \g$ with $[h,e]=2e$, $[h,y]=-2y$,
$[e,f]=h$. The triple $(h,e,f)$ is said to be an $\ssl_2$-triple.
The proof of the Jacobson-Morozov lemma in \cite{jac} translates to
a straightforward algorithm to find such a triple containing a given
nilpotent element $e$, which takes the following steps:

\begin{enumerate}
\item By solving a system of non-homogeneous linear equations we can find
$z,h\in \g$ with $[e,z]=h$ and $[h,e]=2e$.
\item Set $R= C_{\g}(e)$, the centralizer of $e$ in $\g$. Then the map
$\ad h +2 : R\to R$ is non-singular; hence there exists $u_1\in R$ with
$(\ad h +2 )(u_1)=u_0$, where $u_0 = [h,z]+2z$. We find $u_1$ by solving
a non-homogeneous system of linear equations.
\item Set $f=z-u_1$; then $(h,e,f)$ is a $\ssl_2$-triple.
\end{enumerate}

For a proof of the following theorem we refer to \cite{colmcgov},
Chapter 3.

\begin{theorem}\label{thm:class}
Let $e_1,e_2\in \g$ be two nilpotent elements lying in $\ssl_2$-triples
$(h_1,e_1,f_1)$ and $(h_2,e_2,f_2)$. Then $e_1$ and $e_2$ are $G$-conjugate
if and only if the two $\ssl_2$-triples are $G$-conjugate, if and only
if $h_1$ and $h_2$ are $G$-conjugate.
\end{theorem}

\begin{remark}
Of course the elements $z,h$ found in Step (1) of the algorithm
are not necessarily unique.
Indeed, let $u\in C_{\g}(e)\cap [e,\g]$, and let $v\in \g$ be such that
$[e,v]=u$. Then $z' = z+v$, $h'=h+u$ is also a solution. However, because of
Theorem \ref{thm:class}, this non-uniqueness does not lead to problems.
\end{remark}

\subsection{The weighted Dynkin diagram}\label{sect:wd}

Let $e\in \g$ be a nilpotent element lying in an $\ssl_2$-triple $(h,e,f)$.
Then by the representation theory of $\ssl_2$
we get a direct sum decomposition $\g = \oplus_{k\in \Z} \g(h,k)$, where
$\g(h,k) = \{ x\in \g \mid [h,x]=kx \}$.
Fix a Cartan subalgebra $H$ of $\g$ with
$h\in H$. Let $\Phi$ be the corresponding root system of $\g$.
For $\alpha\in \Phi$ we let $x_\alpha$ be a corresponding root vector. For
each $\alpha$ there is a $k\in \Z$ with $x_\alpha\in \g(h,k)$. We write
$\eta(\alpha)=k$. It can be shown that there exists a basis of simple roots
$\Delta\subset \Phi$ such that
$\eta(\alpha)\geq 0$ for all $\alpha\in\Delta$. Furthermore, for such a
$\Delta$ we have $\eta(\alpha)\in \{0,1,2\}$ for all $\alpha\in\Delta$
(see \cite{cart}). Write $\Delta = \{\alpha_1,
\ldots,\alpha_l\}$. Then the Dynkin diagram of $\Phi$ has $l$ nodes, the
$i$-th node corresponding to $\alpha_i$. Now to each node we add the label
$\eta(\alpha_i)$; the result is called the weighted Dynkin diagram. It
depends only on $e$, and not on the choice of
$\ssl_2$-triple containing $e$. Furthermore, it completely identifies the
nilpotent orbit $Ge$. In other words, $e$ and $e'$ ly in the same $G$-orbit
if and only if they have the same weighted Dynkin diagrams. \par
It is possible to formulate an algorithm for computing the weighted Dynkin
diagram of a given nilpotent $e\in \g$. However, computing the set of roots
relative to a Cartan subalgebra $H$ will in many cases be rather time
consuming, and in the worst cases even prove to be infeasible (for example
if $H$ is not split over the rationals, and the construction of field
extensions is required). Therefore we consider a
different approach. Let $e\in\g$ be nilpotent, and let $(h,e,f)$ be an
$\ssl_2$-triple. We consider the direct sum decomposition $\g \oplus_{k\in \Z} \g(h,k)$; and form the sequence
$s(e)=(\dim \g(h,0), \dim \g(h,1),\ldots, \dim \g(h,m) )$, where $m$ is maximal
with $\dim \g(h,m)\neq 0$. We note that it does not depend on the choice of
$h$ (indeed, if $(h',e,f')$ is an $\ssl_2$-triple, then $h$ and $h'$ are
$G$-conjugate by Theorem \ref{thm:class}).
Secondly, for all $e'\in Ge$ we have $s(e') = s(e)$. In other
words, $s(e)$ only depends on the $G$-orbit of $e$. We call $s(e)$ the
signature of the orbit $Ge$.

\begin{proposition}\label{prop:sig}
Let $\g$ be of exceptional type. Then all nilpotent orbits in $\g$ have
a different signature.
\end{proposition}

\begin{proof}
We have verified this statement by a straightforward computer calculation.
A more conceptual argument goes as follows. We note that knowing the signature
of the orbit $Ge$ amounts to knowing the character of the $\ssl_2$-triple
$(h,e,f)$ on the module $\g$. Hence the signature determines the sizes of the
Jordan blocks of the adjoint map $\ad_{\g}(e)$. Now, it is known that each
nilpotent orbit gives rise to a different set of sizes of Jordan blocks
(cf. \cite{lawther}).
\end{proof}

Using Proposition \ref{prop:sig} we
compute weighted Dynkin diagrams in the following way.
We first compute the list of all weighted Dynkin diagrams, with their
corresponding signatures. Then, in order to compute the weighted Dynkin
diagram of a nilpotent element, we compute an $\ssl_2$-triple containing it,
compute the corresponding signature, and look it up in the list. Of course,
the table of signatures is computed only once, and then stored.

We remark that it is straightforward to compute the signature
of a nilpotent orbit, given its weighted Dynkin diagram. Indeed, we form
the vector $\bar{w} =( w_1,\ldots,w_l)$, where $w_i$ is the label corresponding
to $\alpha_i$ in the weighted Dynkin diagram. We write a root $\alpha$ of
$\Phi$ as a linear combination of simple roots
$\alpha = \sum_i a_i \alpha_i$. Let $\bar{a}=(a_1,\ldots,a_l)$.
Then the root space $\g_\alpha$ is contained in $\g(h,k)$
if and only if the inner product  $\bar{w}\cdot \bar{a}$ equals $k$.

\begin{remark}
We can prove the statement of Proposition \ref{prop:sig} also for simple
Lie algebras of type $A_n$ and $C_n$. For type $C_n$ the argument goes as
follows. Here $\g$ is isomorphic to $\mathfrak{sp}_{2n}(\C)$. Let $V$ be
its natural module. Then the adjoint module is isomorphic to the symmetric
square $S^2(V)$. Let $\s$ denote the subalgebra of $\g$ spanned by the
$\ssl_2$-triple $(h,e,f)$. Let $\chi_V(x)$ denote the character of $\s$
on $V$. Then
$$\chi_V(x) = \sum_{k\in \Z} n_k x^k, $$
where $n_k$ is the dimension of the weight space with weight $k$.
We have
$$2\chi_{S^2(V)}(x) = \chi_V(x)^2 + \chi_V(x^2).$$
From this it follows that from the knowledge of the character of
$\s$ on $S^2(V)$ we can recover the character of $\s$ on $V$. But that
determines the $G$-conjugacy class of $h$, and hence the nilpotent orbit
$Ge$. The proof for $A_n$ is simpler, as here $\g$ is isomorphic to
$\ssl_{n+1}(\C)$. In this case, from the character of
$\s$ on $\g$ we directly get the character of $\s$ on $V\otimes V^*$. From
the latter character we recover the character of $\s$ on $V$.\par
However, the statement of the proposition does not hold for $\g$ of
type $D_n$. First of all, if the weighted Dynkin diagrams of two nilpotent
orbits can be transformed into each other by a diagram automorphism, then
the two orbits will have the same signature. Secondly, also in other cases
two different orbits can have the same signature. Indeed, consider the
Lie algebra of type $D_{64}$, and the nilpotent orbits corresponding
to the partitions  $(5,3^{27},1^{42})$ and $(4^8,2^{48})$ (we refer to 
\cite{colmcgov} for an account of the parametrization of nilpotent orbits
by partitions). These two orbits have ``essentially'' different weighted
Dynkin diagrams (that is, they cannot be transformed into each other
by a diagram automorphism), and also have the same signature. We are grateful
to Oksana Yakimova for indicating this example to us.
\end{remark}

\subsection{Levi and parabolic subalgebras}\label{sect:levi}

Let $\Phi$ denote the root system of $\g$, with set of simple roots
$\Delta$. For $\alpha\in\Phi$ let $\g_\alpha$ denote the corresponding
root space, spanned by the root vector $x_\alpha$. \par
A subalgebra $\p$ of $\g$ is said to be parabolic if it contains a
Borel subalgebra (i.e., a maximal solvable subalgebra). Let $\Pi\subset
\Delta$, and let $\p_\Pi$ be the subalgebra of $\g$ generated by
$\h$, $\g_\alpha$ for $\alpha>0$ and $\g_{-\alpha}$, for $\alpha\in \Pi$.
Then $\p_\Pi$ is a parabolic subalgebra. Furthermore, any parabolic
subalgebra of $\g$ is $G$-conjugate to a subalgebra of the form $\p_\Pi$.
These $\p_\Pi$ are called standard parabolic subslgabras. \par
Let $\Psi\subset \Phi$ be the root subsystem generated by $\Pi$, and
let $\p=\p_\Pi$. Then $\p = \ml\oplus \n$,
where $\ml$ is spanned by $\h$ along with $\g_\alpha$ for $\alpha\in \Psi$,
and $\n$ is spanned by $\g_\alpha$ for $\alpha\in \Phi$ such that
$\alpha >0$ and $\alpha\not\in \Psi$. The subalgebra $\ml$ is reductive,
and it is called a (standard) Levi subalgebra. \par
Let $\ml_1$, $\ml_2$ be two Levi subalgebras, of standard parabolic subalgebra
corresponding to $\Pi_1,\Pi_2\subset \Delta$. They may or may not be
$G$-conjugate (even if they are isomorphic as abstract Lie algebras, they
may not be $G$-conjugate). We use the following algorithm to decide
whether they are conjugate:

\begin{enumerate}
\item For $i=1,2$ set
$$u_i = \sum_{ \alpha\in \Pi_i} x_{\alpha}.$$
\item Compute the weighted Dynkin diagrams $D_i$ of $u_i$.
\item If $D_1=D_2$ then $\ml_1$ and $\ml_2$ are $G$-conjugate, otherwise
they are not.
\end{enumerate}

\begin{lemma}
The previous algorithm is correct.
\end{lemma}

\begin{proof}
The $u_i$ are representatives of the principal nilpotent orbit of
$\ml_i$ (\cite{colmcgov}, proof of Theorem 4.1.6). This orbit is
distinguished (this follows from the cited proof, along with
\cite{colmcgov}, Lemma 8.2.1). In other words, $\ml_i$ is the
minimal Levi subalgebra of $\g$ containing $u_i$.  Now suppose that
$D_1=D_2$. Then there is a $g\in G$ with $gu_1=u_2$. So $g\ml_1$ is a
minimal Levi subalgebra containing $u_2$. Hence, by \cite{colmcgov},
Theorem 8.1.1,
there is a $g'\in G$ with $g'g\ml_1 = \ml_2$. The reverse direction is
trivial.
\end{proof}

\section{Induced Nilpotent Orbits}\label{sect:induce}

Let $\p\subset \g$ be a parabolic subalgebra, and write $\p = \ml\oplus
\n$. Let $L\subset G$ be the connected subgroup with Lie algebra $\ml$.
Let $Le_0\subset \ml$ be a nilpotent orbit in $\ml$. The following is part of
the content of \cite{colmcgov}, Theorem 7.1.1.

\begin{theorem}\label{thm:1}
\begin{enumerate}
\item[$(1)$] There is a unique nilpotent orbit $Ge\subset \g$ such that
$Ge \cap (Le_0\oplus \n)$ is dense in $Le_0\oplus \n$.
\item[$(2)$] $\dim Ge = \dim Le_0 + 2\dim \n$.
\item[$(3)$] $Ge$ is the unique nilpotent orbit in $\g$ of that dimension
meeting $Le_0\oplus \n$.
\end{enumerate}
\end{theorem}

The orbit $Ge$ of the theorem is said to be induced from the orbit $Le_0$.
It only depends on the Levi subalgebra $\ml$ and not on the parabolic
subalgebra $\p$ (\cite{colmcgov}, Theorem 7.1.3). Therefore we write
$Ge = \Ind_{\ml}^{\g}(Le_0)$. Also, induction is transitive: if $\ml_1\subset
\ml_2$ are two Levi subalgebras, and $L_1e_1$ is a nilpotent orbit
in $\ml_1$, then
$$ {\Ind}_{\ml_2}^{\g} ({\Ind}_{\ml_1}^{\ml_2}(L_1e_1)) = {\Ind}_{\ml_1}^{\g}
(L_1e_1)$$
(\cite{colmcgov}, Proposition 7.1.4). Furthermore, a nilpotent orbit
in $\g$ that is not induced from a nilpotent orbit of a Levi subalgebra is
said to be rigid.\par
The problem considered in this paper is to determine $\Ind_{\ml}^{\g}(Le_0)$
for all $G$-conjugacy classes of Levi subalgebras $\ml\subset \g$, and
nilpotent orbits $Le_0\subset \ml$. Of course, because of transitivity we
may restrict to the rigid nilpotent orbits of $\ml$.\par
Consider the union of all $G$-orbits in $\g$ of the same dimension $d$. The
irreducible components of these varieties are called the sheets of $\g$ (cf.
\cite{borho}, \cite{borhokraft}).
The sheets of $\g$ are parametrised by the $G$-conjugacy
classes of pairs $(\ml, Le)$, where $\ml$ is a Levi subalgebra, and
$Le$ a rigid nilpotent orbit in $\ml$. Furthermore,
in the sheet corresponding to
$(\ml,Le)$ there is a unique nilpotent orbit, namely $\Ind_{\ml}^{\g}(Le)$.\par
Let $\Pi\subset \Delta$ and let $\p = \p_\Pi$ be the corresponding
standard parabolic subalgebra. Write $\p = \ml\oplus \n$, and let
$Le_0$ be a rigid nilpotent orbit in $\ml$. Let $\D$ be the Dynkin diagram of
$\g$, which we label in the following way. If $\alpha\not\in\Pi$ then we
give the node corresponding to $\alpha$ the label $2$. Let $D_0$ be the
weighted Dynkin diagram of $Le_0$. Since $\Pi$ is a set of simple roots of
$\ml$, the nodes of $D_0$ correspond to the elements of $\Pi$. We give
the node in $\D$ corresponding to $\alpha\in \Pi$ its label in $D_0$,
under this correspondence. We call the diagram $\D$ together with its labeling,
the sheet diagram of the sheet corresponding to $(\ml,Le_0)$. \par
We note that from the sheet diagram we can recover the sheet. Indeed,
from \cite{colmcgov}, Theorem 7.1.6 we get that the weighted Dynkin
diagram of a rigid nilpotent element only has labels $0$ and $1$. Hence,
from the nodes with label $0$ or $1$ we recover $\Pi$. Then from the labels
in those nodes we get the weighted Dynkin diagram of the orbit $Le_0$.
Hence we recover the orbit $Le_0$ as well. \par
Let $\D$ be the sheet diagram corresponding to $(\ml,Le_0)$. Let
$\omega : \Phi^+ \to \Z$ be the additive function such that for $\alpha\in
\Delta$, $\omega(\alpha)$
is the label of the node corresponding to $\alpha$ in $\D$. Let
$\uu(\D)$ be the subspace of $\g$ spanned by all $\g_\alpha$ with
$\omega(\alpha) \geq 2$. Let $Ge = \Ind_{\ml}^{\g}(Le_0)$.

\begin{lemma}\label{lem:1}
$Ge \cap \uu(\D)$ is dense in $\uu(\D)$.
\end{lemma}

\begin{proof}
Recall that $\Pi$ is a set of simple roots of the root system $\Psi$ of $\ml$.
If $e_0=0$ then $\uu(\D) = \n$, and the lemma follows by Theorem \ref{thm:1}.
So now we assume that $e_0\neq 0$. Note that $\ml'=[\ml,\ml]$
is the semisimple part of $\ml$. Let $\h' = \h\cap [\ml,\ml]$; then $\h'$ is a
Cartan subalgebra of $\ml'$. Let $h_0\in \h'$ be such that $\alpha(h_0)$ is
the label in $D$ corresponding to $\alpha$, for $\alpha\in \Pi$. After
possibly replacing $e_0$ by a $L$-conjugate we may assume that $e_0$ lies
in an $\ssl_2$-triple $(h_0,e_0,f_0)$. Let $\ml = \oplus_{k\in \Z} \ml(h_0,k)$
be the corresponding grading of $\ml$. Set $\ml_{\geq 2} = \oplus_{k\geq 2}
\ml(h_0,k)$. Then by \cite{colmcgov}, Lemma 4.1.4, $Le_0 \cap \ml_{\geq 2}$
is dense in $\ml_{\geq 2}$. \par
Let $\beta\in \Phi^+$, but not in $\Psi$. Then written as a linear combination
of simple roots, $\beta$ has at least one $\alpha\in \Delta\setminus \Pi$
with positive coefficient. Hence $\omega(\beta)\geq 2$. Furthemore, a
$\beta\in \Psi$ has $\omega(\beta) = \beta(h_0)$; hence $\omega(\beta)
\geq 2$ if and only if $\g_\beta\subset \ml_{\geq 2}$. It follows that
$\uu(\D)=\ml_{\geq 2} \oplus \n$.  \par
We conclude that $Le_0\oplus \n\cap \uu(\D)$ is dense in $\uu(\D)$. Since
$Ge\cap Le_0\oplus \n$ is dense in $Le_0\oplus \n$, we get that
$Ge \cap \uu(\D)$ is dense in $\uu(\D)$.
\end{proof}

\begin{lemma}\label{lem:2}
Write $s=\dim Ge$, and let $e'\in \uu(\D)$. Then $e'\in Ge$ if and only if
$\dim C_{\g}(e') = \dim\g -s$.
\end{lemma}

\begin{proof}
First of all, if $e'\in Ge$ then $\dim C_{\g}(e') = \dim C_{\g}(e) =
\dim \g-s$. For the converse, let $t$ be the minimum dimension
of a centralizer $C_{\g}(u)$, for $u\in \uu(\D)$. Then the set of elements
of $\uu(\D)$ with centralizer of dimension $t$ is dense in $\uu(\D)$. But
two dense sets always meet. Hence from lemma \ref{lem:1} we get
$t = \dim C_{\g}(e) = \dim \g -s$.
So if $e'\in \uu(\D)$, and $\dim C_{\g}(e') = \dim \g-s$,
then the dimension of $C_{\g}(e')$ is minimal among all elements of $\uu(\D)$.
Hence the dimension of the orbit $Ge'$ is maximal. Now from \cite{borhokraft},
Satz 5.4 it follows that $e'$ lies in the sheet corresponding to $\D$.
Therefore, $e'\in Ge$.
\end{proof}

Now we describe an algorithm for listing the induced nilpotent orbits in $\g$.
The output is a list of pairs $(D,\widetilde{D})$, where $\widetilde{D}$
runs through the sheet diagrams of $\g$, and $D$ is the weighted Dynkin diagram
of the corresponding nilpotent orbit. We use a set $\Omega = \{0,1,\ldots,N\}$
of integers; where $N$ is a previously chosen parameter. The algorithm
takes the following steps:

\begin{enumerate}
\item Using the method of Section \ref{sect:levi} get representatives of
the $G$-conjugacy classes of Levi subalgebras, each lying inside a
standard parabolic subalgebra.
\item Set $\mathcal{I}=\varnothing$.
For each Levi subalgebra $\ml$ from the list, and each rigid
nilpotent orbit $Le_0\subset \ml$ do the following:
\begin{enumerate}
\item Construct the sheet diagram $\widetilde{D}$ of the pair $(\ml,Le_0)$.
\item Compute a basis $\{u_1,\ldots,u_m\}$ of the space $\uu(\widetilde{D})$.
\item Set $s = \dim Le_0 +2\dim \n$ (where $\p = \ml\oplus \n$ is the
parabolic subalgebra containing $\ml$).
\item Let $e'=\sum_{i=1}^m \alpha_iu_i\in \uu(\widetilde{D})$, where the
$\alpha_i\in \Omega$ are chosen randomly, uniformly, and independently.
\item
If $\dim C_{\g}(e') = \dim \g-s$ then
compute the weighted Dynkin diagram (cf. Section \ref{sect:wd})
$D$ of the orbit $Ge'$
and add $(D,\widetilde{D})$ to $\mathcal{I}$. Otherwise return to Step 2(d).
\end{enumerate}
\item Return $\mathcal{I}$.
\end{enumerate}

\begin{proposition}
The previous algorithm is correct, and terminates for large enough $N$.
\end{proposition}

\begin{proof}
Let $\widetilde{D}$ be a sheet diagram, and let $Ge$ be the corresponding
induced nilpotent orbit. Then by Lemma \ref{lem:2},
steps 2(d) and 2(e) are executed until an element $e'$ of $Ge$ is found.
So, if the algorithm terminates, then it returns the correct output.
On the other hand, by Lemma \ref{lem:1}, the set $Ge\cap \uu(\widetilde{D})$
is dense in $\uu(\widetilde{D})$. Hence for large enough $N$, the random
element $e'$ lies in $Ge$ with high probability. Therefore, the algorithm
will terminate.
\end{proof}

\begin{remark}
In practice, it turns out that selecting a rather small $N$ (e.g., $N=10$)
suffices in order that the algorithm terminates. Furthermore, if the
algorithm needs to many rounds for a given $N$, then one can try again
with a higher value for $N$.
\end{remark}

\bigskip
\section{The Tables}\label{sect:tables}

In this section we give the tables of the induced nilpotent orbits in the
Lie algebras of exceptional type, computed with the algorithm of the previous
section.\footnote{The program needed 17, 282, 9055, 3 and 0.1
seconds respectively for $E_6$, $E_7$, $E_8$, $F_4$ and $G_2$;
the computations were done on a
2GHz processor with 1GB of memory for {\sf GAP}.}
In order to use this algorithm, we need to know the rigid nilpotent
orbits of each Levi subalgebra. For Levi subalgebras of classical type,
this is described in \cite{colmcgov}, \S 7.3. For the Lie algebras of
exceptional type it follows from our calculations what the rigid nilpoten
orbits are. We summarise this in the following theorem.

\begin{theorem}
The weighted Dynkin diagrams of the rigid nilpotent orbits $($except the zero
orbit$)$ in $E_6$, $E_7$, $E_8$, $F_4$ and $G_2$ are given in Tables
{\rm \ref{tab:rigidE6}, \ref{tab:rigidE7}, \ref{tab:rigidE8}, \ref{tab:rigidF4},~\ref{tab:rigidG2}.}
\end{theorem}

\setlongtables

\begin{longtable}{|l|c|}
\caption{Rigid nilpotent orbits in $E_6$.}\label{tab:rigidE6}
\endfirsthead
\hline
\multicolumn{2}{|l|}{\small\slshape Rigid nilpotent orbits in $E_6$.} \\
\hline
\endhead
\hline
\endfoot
\endlastfoot

\hline

label & characteristic \\
\hline

$A_1$ & $0~~~~0~~~~\overset{\text{\normalsize 1}}{0}~~~~0~~~~0$ \\

$3A_1$ & $0~~~~0~~~~\overset{\text{\normalsize 0}}{1}~~~~0~~~~0$ \\

$2A_2+A_1$ & $1~~~~0~~~~\overset{\text{\normalsize 0}}{1}~~~~0~~~~1$ \\

\hline

\end{longtable}

\begin{longtable}{|l|c|}
\caption{Rigid nilpotent orbits in $E_7$.}\label{tab:rigidE7}
\endfirsthead
\hline
\multicolumn{2}{|l|}{\small\slshape Rigid nilpotent orbits in $E_7$.} \\
\hline
\endhead
\hline
\endfoot
\endlastfoot

\hline
label & characteristic \\
\hline

$A_1$ & $1~~~~0~~~~\overset{\text{\normalsize 0}}{0}~~~~0~~~~0~~~~0$ \\

$2A_1$ & $0~~~~0~~~~\overset{\text{\normalsize 0}}{0}~~~~0~~~~1~~~~0$ \\

$(3A_1)'$ & $0~~~~1~~~~\overset{\text{\normalsize 0}}{0}~~~~0~~~~0~~~~0$ \\

$4A_1$ & $0~~~~0~~~~\overset{\text{\normalsize 1}}{0}~~~~0~~~~0~~~~1$ \\

$A_2+2A_1$ & $0~~~~0~~~~\overset{\text{\normalsize 0}}{1}~~~~0~~~~0~~~~0$ \\

$A_1+2A_2$ & $0~~~~1~~~~\overset{\text{\normalsize 0}}{0}~~~~0~~~~1~~~~0$ \\

$(A_1+A_3)'$ & $1~~~~0~~~~\overset{\text{\normalsize 0}}{1}~~~~0~~~~0~~~~0$ \\

\hline

\end{longtable}

\begin{longtable}{|l|c|}
\caption{Rigid nilpotent orbits in $E_8$.}\label{tab:rigidE8}
\endfirsthead
\hline
\multicolumn{2}{|c|}{\small\slshape Rigid nilpotent orbits in $E_8$.} \\
\hline
\endhead
\hline
\endfoot
\endlastfoot

\hline
label & characteristic \\
\hline

$A_1$ &
$0~~~~0~~~~\overset{\text{\normalsize 0}}{0}~~~~0~~~~0~~~~0~~~~1$ \\

$2A_1$ &
$1~~~~0~~~~\overset{\text{\normalsize 0}}{0}~~~~0~~~~0~~~~0~~~~0$\\

$3A_1$ &
$0~~~~0~~~~\overset{\text{\normalsize 0}}{0}~~~~0~~~~0~~~~1~~~~0$\\

$4A_1$ &
$0~~~~0~~~~\overset{\text{\normalsize 1}}{0}~~~~0~~~~0~~~~0~~~~0$\\

$A_2+A_1$ &
$1~~~~0~~~~\overset{\text{\normalsize 0}}{0}~~~~0~~~~0~~~~0~~~~1$\\

$A_2+2A_1$ &
$0~~~~0~~~~\overset{\text{\normalsize 0}}{0}~~~~0~~~~1~~~~0~~~~0$\\

$A_2+3A_1$ &
$0~~~~1~~~~\overset{\text{\normalsize 0}}{0}~~~~0~~~~0~~~~0~~~~0$\\

$2A_2+A_1$ &
$1~~~~0~~~~\overset{\text{\normalsize 0}}{0}~~~~0~~~~0~~~~1~~~~0$\\

$A_3+A_1$ &
$0~~~~0~~~~\overset{\text{\normalsize 0}}{0}~~~~0~~~~1~~~~0~~~~1$\\

$2A_2+2A_1$ &
$0~~~~0~~~~\overset{\text{\normalsize 0}}{0}~~~~1~~~~0~~~~0~~~~0$\\

$A_3+2A_1$ &
$0~~~~1~~~~\overset{\text{\normalsize 0}}{0}~~~~0~~~~0~~~~0~~~~1$\\

$D_4(a_1)+A_1$ &
$0~~~~0~~~~\overset{\text{\normalsize 1}}{0}~~~~0~~~~0~~~~1~~~~0$\\

$A_3+A_2+A_1$ &
$0~~~~0~~~~\overset{\text{\normalsize 0}}{1}~~~~0~~~~0~~~~0~~~~0$\\

$2A_3$ &
$1~~~~0~~~~\overset{\text{\normalsize 0}}{0}~~~~1~~~~0~~~~0~~~~0$ \\

$D_5(a_1)+A_2$ &
$0~~~~1~~~~\overset{\text{\normalsize 0}}{0}~~~~0~~~~1~~~~0~~~~1$ \\

$A_5+A_1$ &
$1~~~~0~~~~\overset{\text{\normalsize 0}}{1}~~~~0~~~~0~~~~0~~~~1$ \\

$A_4+A_3$ &
$0~~~~0~~~~\overset{\text{\normalsize 0}}{1}~~~~0~~~~0~~~~1~~~~0$ \\

\hline

\end{longtable}

\begin{longtable}{|l|c|}
\caption{Rigid nilpotent orbits in $F_4$.}\label{tab:rigidF4}
\endfirsthead
\hline
\multicolumn{2}{|l|}{\small\slshape Rigid nilpotent orbits in $F_4$.} \\
\hline
\endhead
\hline
\endfoot
\endlastfoot

\hline

label & characteristic \\
\hline
 &
\begin{picture}(20,7)
  \put(-20,0){\circle*{6}}
  \put(0,0){\circle*{6}}
  \put(20,0){\circle{6}}
  \put(40,0){\circle{6}}
  \put(-17,0){\line(1,0){14}}
  \put(2,2){\line(1,0){16}}
  \put(2,-2){\line(1,0){16}}
  \put(23,0){\line(1,0){14}}
\end{picture}

\\
\hline

$A_1$ & 1~~~0~~~0~~~0\\

$\widetilde{A}_1$ & 0~~~0~~~0~~~1\\

$A_1+\widetilde{A}_1$ & 0~~~1~~~0~~~0\\

$A_2+\widetilde{A}_1$ & 0~~~0~~~1~~~0\\

$\widetilde{A}_2+A_1$ & 0~~~1~~~0~~~1\\

\hline

\end{longtable}

\begin{longtable}{|l|c|}
\caption{Rigid nilpotent orbits in $G_2$.}\label{tab:rigidG2}
\endfirsthead
\hline
\multicolumn{2}{|l|}{\small\slshape Rigid nilpotent orbits in $G_2$.} \\
\hline
\endhead
\hline
\endfoot
\endlastfoot

\hline

label & characteristic \\

&

  \begin{picture}(10,7)
  \put(-8,0){\circle*{6}}
  \put(22,0){\circle{6}}
  \put(-8,-2){\line(1,0){28}}
  \put(-5,0){\line(1,0){24}}
  \put(-8,2){\line(1,0){28}}
\end{picture}                    \\
\hline

\hline

$A_1$ & 1~~0\\

$\widetilde{A}_1$ & 0~~1 \\

\hline

\end{longtable}

The tables with the induced orbits have one row for each sheet.
There are six columns. In the first column we give the label of
the induced orbit corresponding to the sheet (where we use the
same labels as in \cite{colmcgov}). The second and third columns
contain, respectively, the dimension and the weighted Dynkin diagram
$D$ (here called {\em characteristic}) of the induced nilpotent orbit.
The dimensions of the nilpotent orbits are well-known;
they were calculated in \cite{elas}, and are also contained in the tables
of \cite{colmcgov}. The fourth column has the sheet diagram $\D$. The fifth
column contains the rank of the sheet.  This notion is defined as follows.
Let the sheet correspond to $(\ml,Le)$, where $\ml$ is a Levi subalgebra, and
$Le$ a rigid nilpotent orbit in it. Then the rank of the sheet is the dimension
of the center of $\ml$. It is straightforward to see that this equals the
number of labels 2 in the sheet diagram $\D$. Finally, in the last column we
give a representative of the induced orbit, i.e., an element of $\uu(\D)$
with weighted Dynkin diagram $D$. Such a representative $e$ is
given as a sum of positive root vectors, $e=x_{\beta_1}+\cdots +x_{\beta_r}$.
Then to $e$ there corresponds a Dynkin diagram, which is simply the Dynkin
diagram of the roots $\beta_i$.
This diagram has $r$ nodes, and node $i$ is connected
to node $j$ by $\langle \beta_i, \beta_j^\vee\rangle \langle \beta_j,
\beta_i^\vee\rangle =0,1,2,3$ lines. Furthermore, if these scalar products are
positive, then the lines are dotted. Finally, if the root $\beta_i$ is
long, then node $i$ is black. For each
representative we give the corresponding Dynkin diagram,
where each node has a numerical label, which denotes the
position of the corresponding positive root as used by {\sf GAP}4
(\cite{gra14} contains lists of those roots).
In other words, if the
Dynkin diagram of a representative has labels
$i_1,\ldots,i_k$, then the corresponding representative is the sum of the root
vectors corresponding to the $i_j$-th positive root (in the order in
which they appear in {\sf GAP}4) for $1\leq j\leq k$. \par
These representatives have been found as folllows. First of all, for each
nilpotent orbit one or more Dynkin diagrams of representatives are known
(some are described in \cite{gra14}, many have been found by Elashvili).
For an induced nilpotent orbit with weighted Dynkin diagram $D$ and
sheet diagram $\widetilde{D}$ we construct the set $S$ of roots
$\alpha >0$ such that
$\g_\alpha$ is contained in $\uu(\widetilde{D})$. We have written a
simple-minded
program in {\sf GAP}4 that, for a given Dynkin diagram of a representative,
tries to find a subset
of $S$ such that its Dynkin diagram is the given one (basically by trying all
possibilities). We executed this program for all known Dynkin diagrams
of representatives of the nilpotent orbit. In all cases we managed to
find a representative this way. Furthermore, the element found was
shown to be a representative of the nilpotent orbit by checking that its
weighted Dynkin diagram was equal to $D$.

We note that in our tables the first three columns contain information
relative to the induced orbit, wheras the last three columns contain
information about the sheet. On some occasions it happens that a nilpotent
orbit is induced in more than one way (i.e., it occurs in more than one
sheet). In these cases we have not repeated the information in the first
three columns; instead we have grouped the rows in the last three columns
together by using a curly brace.

\begin{longtable}{|l|c|c|c|c|l|}
\caption{Induced nilpotent orbits in the Lie algebra of type $E_6$.}
\endfirsthead
\hline
\multicolumn{6}{|l|}{\small\slshape Induced orbits in type $E_6$} \\
\hline
\endhead
\hline
\endfoot
\endlastfoot

\hline
label & dim & characteristic & sheet diagram & rk & representative\\
\hline

\nm{$E_6$} & 72 &
$2~~~~2~~~~\overset{\text{\normalsize 2}}{2}~~~~2~~~~2$ &
$2~~~~2~~~~\overset{\text{\normalsize 2}}{2}~~~~2~~~~2$ &
6 &
\begin{picture}(40,30)
  \put(3,0){\circle{6}}
  \put(5,4){\lab{1}}
  \put(23,0){\circle{6}}
\put(25,4){\lab{3}}
  \put(43,0){\circle{6}}
\put(45,4){\lab{4}}
  \put(63,0){\circle{6}}
\put(65,4){\lab{5}}
  \put(83,0){\circle{6}}
\put(85,4){\lab{6}}
  \put(43,20){\circle{6}}
\put(47,21){\lab{2}}
  \put(6,0){\line(1,0){14}}
  \put(26,0){\line(1,0){14}}
  \put(46,0){\line(1,0){14}}
  \put(66,0){\line(1,0){14}}
  \put(43,3){\line(0,1){14}}
\end{picture}   \\

\nm{$E_6(a_1)$} & 70 &
$2~~~~2~~~~\overset{\text{\normalsize 2}}{0}~~~~2~~~~2$ &
$2~~~~2~~~~\overset{\text{\normalsize 2}}{0}~~~~2~~~~2$ &
5 &
\begin{picture}(25,40)
\put(33,0){\circle{6}}
\put(22,-2){\lab{8}}
\put(53,0){\circle{6}}
\put(56,-2){\lab{2}}
\dashline[20]{2}(36,0)(50,0)
\put(33,3){\line(0,1){14}}
\put(53,3){\line(0,1){14}}
\put(36,20){\line(1,0){14}}
\put(33,20){\circle{6}}
\put(24,23){\lab{3}}
\put(53,20){\circle{6}}
\put(55,23){\lab{10}}
\put(4,20){\circle{6}}
\put(7,20){\line(1,0){23}}
\put(3,25){\lab{1}}
\put(73,20){\circle{6}}
\put(56,20){\line(1,0){14}}
\put(75,23){\lab{6}}
\end{picture}
\\

\nm{$D_5$} & 68 &
$2~~~~0~~~~\overset{\text{\normalsize 2}}{2}~~~~0~~~~2$ &
$2~~~~0~~~~\overset{\text{\normalsize 2}}{2}~~~~0~~~~2$ &
4 &
\begin{picture}(50,35)
\put(3,0){\circle{6}}
\put(5,5){\lab{4}}
\put(23,0){\circle{6}}
\put(25,5){\lab{2}}
\put(6,0){\line(1,0){14}}
\put(26,0){\line(1,0){14}}
\put(43,0){\circle{6}}
\put(45,5){\lab{15}}
\put(63,0){\circle{6}}
\put(46,0){\line(1,0){14}}
\put(65,5){\lab{1}}
\put(43,20){\circle{6}}
\put(43,3){\line(0,1){14}}
\put(45,24){\lab{6}}
\end{picture}
\\
\hline
\end{longtable}

{\footnotesize
\begin{longtable}{|l|c|c|c|c|l|}
\caption{\fontsize{12}{12pt}\selectfont Induced nilpotent orbits in the Lie algebra of type~$E_7$.}
\endfirsthead
\hline
\multicolumn{6}{|l|}{\small\slshape Induced orbits in type $E_7$} \\
\hline
\endhead
\hline
\endfoot
\endlastfoot

\hline
label & dim & characteristic & sheet diagram & rk & representative \\
\hline

\nm{$E_7$} & 126 &
$2~~~~2~~~~\overset{\text{\normalsize 2}}{2}~~~~2~~~~2~~~~2$ &
$2~~~~2~~~~\overset{\text{\normalsize 2}}{2}~~~~2~~~~2~~~~2$ &
7 &
\begin{picture}(100,40)
  \put(3,0){\circle{6}}
  \put(5,4){\lab{1}}
  \put(23,0){\circle{6}}
\put(25,4){\lab{3}}
  \put(43,0){\circle{6}}
\put(45,4){\lab{4}}
  \put(63,0){\circle{6}}
\put(65,4){\lab{5}}
  \put(83,0){\circle{6}}
\put(85,4){\lab{6}}
  \put(43,20){\circle{6}}
\put(47,21){\lab{2}}
  \put(6,0){\line(1,0){14}}
  \put(26,0){\line(1,0){14}}
  \put(46,0){\line(1,0){14}}
  \put(66,0){\line(1,0){14}}
  \put(43,3){\line(0,1){14}}
\put(103,0){\circle{6}}
\put(86,0){\line(1,0){14}}
\put(105,5){\lab{7}}
\end{picture}   \\

\nm{$E_7(a_1)$} & 124 &
$2~~~~2~~~~\overset{\text{\normalsize 2}}{0}~~~~2~~~~2~~~~2$ &
$2~~~~2~~~~\overset{\text{\normalsize 2}}{0}~~~~2~~~~2~~~~2$ &
6 &
\begin{picture}(100,40)
\put(4,20){\circle{6}}
\put(4,25){\lab{7}}
\put(7,20){\line(1,0){14}}
\put(53,0){\circle{6}}
\put(40,-2){\lab{9}}
\put(73,0){\circle{6}}
\put(78,-2){\lab{2}}
\dashline[20]{2}(56,0)(70,0)
\put(53,3){\line(0,1){14}}
\put(73,3){\line(0,1){14}}
\put(56,20){\line(1,0){14}}
\put(53,20){\circle{6}}
\put(44,24){\lab{5}}
\put(73,20){\circle{6}}
\put(75,23){\lab{10}}
\put(24,20){\circle{6}}
\put(27,20){\line(1,0){23}}
\put(23,25){\lab{6}}
\put(93,20){\circle{6}}
\put(76,20){\line(1,0){14}}
\put(95,25){\lab{1}}
\end{picture}
\\
\hline
\end{longtable}}

{\footnotesize
\begin{longtable}{|l@{\,\,}|@{\,\,}c@{\,\,}|@{\,\,}c@{\,\,}|@{\,\,}c@{\,\,}|@{\,\,}c@{\,\,}|l@{\,\,\,}|}
\caption{{\fontsize{12}{12pt}\selectfont Induced nilpotent orbits in the Lie algebra of type~$E_8$.}}
\endfirsthead
\hline
\multicolumn{6}{|l|}{\small\slshape Induced orbits in type $E_8$} \\
\hline
\endhead
\hline
\endfoot
\endlastfoot

\hline
label & dim & characteristic & sheet diagram & rk & representative \\
\hline

\nm{$E_8$} & 240 &
$2~~~~2~~~~\overset{\text{\normalsize 2}}{2}~~~~2~~~~2~~~~2~~~~2$ &
$2~~~~2~~~~\overset{\text{\normalsize 2}}{2}~~~~2~~~~2~~~~2~~~~2$ &
8 &
\begin{picture}(100,40)
  \put(3,0){\circle{6}}
  \put(5,4){\lab{1}}
  \put(23,0){\circle{6}}
\put(25,4){\lab{3}}
  \put(43,0){\circle{6}}
\put(45,4){\lab{4}}
  \put(63,0){\circle{6}}
\put(65,4){\lab{5}}
  \put(83,0){\circle{6}}
\put(85,4){\lab{6}}
  \put(43,20){\circle{6}}
\put(47,21){\lab{2}}
  \put(6,0){\line(1,0){14}}
  \put(26,0){\line(1,0){14}}
  \put(46,0){\line(1,0){14}}
  \put(66,0){\line(1,0){14}}
  \put(43,3){\line(0,1){14}}
\put(103,0){\circle{6}}
\put(86,0){\line(1,0){14}}
\put(105,5){\lab{7}}
\put(123,0){\circle{6}}
\put(125,5){\lab{8}}
\put(106,0){\line(1,0){14}}
\end{picture}   \\

\nm{$E_8(a_1)$} & 238 &
$2~~~~2~~~~\overset{\text{\normalsize 2}}{0}~~~~2~~~~2~~~~2~~~~2$ &
$2~~~~2~~~~\overset{\text{\normalsize 2}}{0}~~~~2~~~~2~~~~2~~~~2$ &
7 &
\begin{picture}(100,40)
\put(4,20){\circle{6}}
\put(4,25){\lab{8}}
\put(7,20){\line(1,0){14}}
\put(24,20){\circle{6}}
\put(24,25){\lab{7}}
\put(27,20){\line(1,0){14}}
\put(73,0){\circle{6}}
\put(58,-2){\lab{10}}
\put(93,0){\circle{6}}
\put(98,-2){\lab{2}}
\dashline[20]{2}(76,0)(90,0)
\put(73,3){\line(0,1){14}}
\put(93,3){\line(0,1){14}}
\put(76,20){\line(1,0){14}}
\put(73,20){\circle{6}}
\put(64,24){\lab{5}}
\put(93,20){\circle{6}}
\put(95,25){\lab{11}}
\put(44,20){\circle{6}}
\put(47,20){\line(1,0){23}}
\put(43,25){\lab{6}}
\put(113,20){\circle{6}}
\put(96,20){\line(1,0){14}}
\put(115,25){\lab{1}}
\end{picture}
\\
\hline
\end{longtable}}

\begin{longtable}{|l|c|c|c|c|l|}
\caption{Induced nilpotent orbits in the Lie algebra of type $G_2$.}
\endfirsthead
\hline
\multicolumn{6}{|l|}{\small\slshape Induced orbits in type $G_2$} \\
\hline
\endhead
\hline
\endfoot
\endlastfoot

\hline
label & dim & characteristic & sheet diagram & rk & representative\\
\hline
 & & &

  \begin{picture}(10,7)
  \put(-8,0){\circle*{6}}
  \put(22,0){\circle{6}}
  \put(-8,-2){\line(1,0){28}}
  \put(-5,0){\line(1,0){24}}
  \put(-8,2){\line(1,0){28}}
\end{picture}

& & \\

\hline

\nm{$G_2$} & 12 &
2~~~~2 &
2~~~~2 &
2 &
\begin{picture}(10,17)
  \put(0,0){\circle{6}}
  \put(0,5){\lab{1}}
  \put(30,0){\circle*{6}}
  \put(32,5){\lab{2}}
  \put(2,-2){\line(1,0){28}}
  \put(3,0){\line(1,0){24}}
  \put(2,2){\line(1,0){28}}
\end{picture} \\

\nm{$A_1+\widetilde{A}_1$} & 10 &
2~~~~0 &

$
\left\{
\begin{aligned}
2~~~~0\\ \\
0~~~~2
\end{aligned}
\right.
$
&
$
\begin{aligned}
1\\ \\
1
\end{aligned}
$
&
$
\begin{aligned}
&
\begin{picture}(5,15)
\put(3,0){\circle*{6}}
\put(5,5){\lab{2}}
\put(23,0){\circle{6}}
\put(25,5){\lab{4}}
\end{picture} \\
& \begin{picture}(5,15)
\put(3,0){\circle*{6}}
\put(5,5){\lab{6}}
\put(23,0){\circle{6}}
\put(25,5){\lab{1}}
\end{picture}
\end{aligned}
$
\\

\hline
\end{longtable}

\begin{longtable}{|l|c|c|c|c|l|}
\caption{Induced nilpotent orbits in the Lie algebra of type $F_4$.}
\endfirsthead
\hline
\multicolumn{6}{|l|}{\small\slshape Induced orbits in type $F_4$} \\
\hline
\endhead
\hline
\endfoot
\endlastfoot

\hline
label & dim & characteristic & sheet diagram & rk & representative\\
\hline
 & & &
%\multicolumn{3}{|l|}{
\begin{picture}(60,7)
  \put(1,0){\circle*{6}}
  \put(21,0){\circle*{6}}
  \put(41,0){\circle{6}}
  \put(61,0){\circle{6}}
  \put(4,0){\line(1,0){14}}
  \put(23,2){\line(1,0){16}}
  \put(23,-2){\line(1,0){16}}
  \put(44,0){\line(1,0){14}}
\end{picture}  & & \\

\hline

\nm{$F_4$} & 48 &
 2~~~2~~~2~~~2 & 2~~~2~~~2~~~2 &
4 &
\begin{picture}(70,15)
  \put(0,0){\circle{6}}
  \put(3,5){\lab{1}}
  \put(20,0){\circle{6}}
\put(23,5){\lab{3}}
  \put(40,0){\circle*{6}}
\put(43,5){\lab{4}}
  \put(60,0){\circle*{6}}
\put(63,5){\lab{2}}
  \put(3,0){\line(1,0){14}}
  \put(22,2){\line(1,0){16}}
  \put(22,-2){\line(1,0){16}}
  \put(43,0){\line(1,0){14}}
\end{picture}    \\
\nm{$F_4(a_1)$}
&
46
&
2~~~2~~~0~~~2
&
$
\left\{
\begin{aligned}
2~~~2~~~2~~~0\\
\\
0~~~2~~~2~~~2
\end{aligned}
\right.
$
&
$
\begin{aligned}
3\\
\\
3
\end{aligned}
$
&
$
\begin{aligned}
&\begin{picture}(5,15)
\put(3,0){\circle{6}}
\put(5,5){\lab{3}}
\put(23,0){\circle*{6}}
\put(25,5){\lab{4}}
\put(6,-1){\line(1,0){14}}
\put(6,1){\line(1,0){14}}
\put(40,0){\circle*{6}}
\put(43,5){\lab{2}}
\put(26,0){\line(1,0){11}}
\put(57,0){\circle*{6}}
\put(60,5){\lab{15}}
\put(43,0){\line(1,0){11}}
\end{picture}\\
&\begin{picture}(50,30)
\put(30,10){\line(-2,1){17}}
\put(30,10){\line(-2,-1){17}}
\put(10,20){\circle{6}}
\put(10,0){\circle{6}}
\dashline[20]{2}(10,17)(10,3)
\put(0,20){\lab{1}}
\put(-2,0){\lab{11}}
\put(33,10){\circle{6}}
\put(35,15){\lab{3}}
\put(53,10){\circle*{6}}
\put(55,15){\lab{4}}
\put(36,9){\line(1,0){14}}
\put(36,11){\line(1,0){14}}
\end{picture}
\end{aligned}
$
\\
\hline
\end{longtable}

\section{Concluding Remarks}

\begin{enumerate}[$5.1$]

\item We remark that from the tables several things can be read off.
For example, a sheet is said to be a {\em Dixmier sheet} if it contains
semisimple elements. We recall that an orbit which is induced
from the zero orbit of a Levi subalgebra is called a
{\em Richardson orbit}. In other words, for a Richardson orbit
$Ge$ there exists a parabolic subalgebra $\p\subset \g$, with nilradical
$\n$ such that $Ge\cap \n$ is dense in $\n$. In this case elements of
$Ge\cap  \n$ are called {\em Richardson elements} of $\n$.
It is known (cf. \cite{borho}) that a sheet is
Dixmier if and only if it corresponds to a Richardson orbit.
In other words, a sheet is Dixmier if and only if the sheet diagram
does not contain any labels 1.

\item Let $x\in \g$ be nilpotent. A parabolic subgroup $P\subset G$ is
said to be a {\em polarization} of $x$ if $x$ is a Richarson element
of the nilradical of the Lie algebra $\p$ of $P$. In \cite{hesselink} and
\cite{kempken} all polarizations of the Richardson orbits are determined,
in Lie algebras of classical type. Our tables give all polarizations of
the Richardson orbits of the Lie algebras of exceptional type.

\item Furthermore, we note that the dimension of a sheet is easily determined
from our tables. Indeed, from \cite{borhokraft}, \S 5.7, Korollar (c), it
follows that the dimension of the sheet is equal to the dimension of
the induced orbit plus the rank of the sheet. So we get the dimension
of the sheet by adding the numbers in the second and fifth columns.

\item
In \cite{carter5}, \cite{cartelk} the diagram of a nilpotent element is
called admissible if
\begin{itemize}
\item the roots corresponding to the nodes are linearly independent,
\item every cycle in the diagram has an even number of nodes.
\end{itemize}
By inspection it can be seen that all diagrams in our tables are
admissible, except in five cases: one for $E_7$, three for $E_8$
and one for $F_4$ (in which cases the diagram has a cycle with an odd
number of nodes).

\item In \cite{baur}, \cite{baur_goodwin} it was shown that every Richardson
orbit in a Lie algebra of classical type (over a field of good characteristic)
has a representative $x$, lying in the nilradical of a parabolic
subalgebra, with
$$ x = \sum_{\alpha\in\Gamma} x_\alpha,$$
where $\Gamma \subset\Phi$. Furthermore, it was shown that the
representative $x$ can be chosen such that the size of $\Gamma$ is equal to
$\rank \g$ minus the dimension of a maximal torus of $C_{\g}(x)$
(which is the minimal size $\Gamma$ can have). By going through the tables
given here it is readily verified that this same statement holds for
all induced nilpotent orbits (and hence for all Richardson orbits)
of Lie algebras of exceptional type, in characteristic $0$. We believe that
this also holds for all induced orbits in Lie algebras of classical
type (and some hand and computer calculations support this). A proof of this
is beyond the scope of the present paper, and will be a theme for further
research.

\end{enumerate}

\section*{Acknowledgement}

Alexander Elashvili was financially supported by the GNSF
(grant \# ST07/3-174). He is also grateful for the hospitality and support
of the University of Bochum (grant SFB/TR12) and the CIRM Trento.

\def\cprime{$'$} \def\Dbar{\leavevmode\lower.6ex\hbox to 0pt{\hskip-.23ex
  \accent"16\hss}D} \def\cprime{$'$} \def\cprime{$'$} \def\cprime{$'$}

\vskip+0.3cm

\centerline{(Received 11.05.2009)}

\vskip+0.2cm

Authors' addresses:

\vskip+0.2cm

Willem A. de Graaf

Dipartimento di Matematica

Universitry of Trento

Via Sommarive 14

38050 Povo (Trento)

Italy

E-mail: degraaf@science.unitn.it

\vskip+0.2cm

Alexander G. Elashvili

A. Razmadze Mathematical Institute

1, M. Alexidze St., Tbilisi 0193

Georgia

E-mail: alela@acnet.ge

\end{document}